\documentclass[12pt,reqno]{amsart}

\usepackage{latexsym}
\usepackage{amssymb}
\usepackage{graphics}
\usepackage{graphicx}
\usepackage{amscd}
\usepackage{MnSymbol}
\usepackage{tikz}
\usepackage{caption}
\usetikzlibrary{positioning}
\usepackage[small,nohug,heads=littlevee]{diagrams}
\diagramstyle[labelstyle=\scriptstyle]

\renewcommand{\epsilon}{\varepsilon}

\newcommand{\PP}{{\mathbb P}}

\newcommand{\C}{{\mathbb C}}

\newcommand{\Z}{{\mathbb Z}}

\newcommand{\CP}{\C\PP}

\renewcommand{\phi}{\varphi}

\newtheorem{theo}{{Theorem}}[section]

\newtheorem{lem}[theo]{{Lemma}}

\newenvironment{example}{\medskip\noindent{\it Example:\/} }{\medskip}
\newenvironment{rem}{\medskip\noindent{\it Remark:\/} }{\medskip}
\newenvironment{Nota}{\medskip\noindent{\it Notation Conventions:\/} }{\medskip}
\newtheorem{defin}[theo]{{\sc Definition}}
\newenvironment{claim}{\medskip\noindent{\it Claim:\/} }{\medskip}

\title{A new method in the Jacobian Conjecture}

\author{Jingzhou Sun }
\address{Department of Mathematics, Stony Brook University, Stony Brook, NY
11777, USA} \email{jsun@math.stonybrook.edu}

\date{\today}

\begin{document}

\begin{abstract}
We introduce a new method in the attempt to prove the Jacobian conjecture. In the complex dimension 2 case, we apply this method to prove some new results related the Jacobian conjecture.
\end{abstract}

\maketitle

 \tableofcontents
 \section{Introduction}
 
The Jacobian conjecture has been a long-lasting open problem.
In the complex case, the conjecture simply asks the following question: A polynomial map from $\C^n$ to itself with non-vanishing Jacobian must be an automorphism? We refer the readers to \cite{bcw} and \cite{essen} for a very good introduction to the history and developments of this problem.  For simplicity, we will call a polynomial map  $F:\C^n\rightarrow \C^n$ a Jacobian map if it has non-vanishing Jacobian.

Even in the dimension 2 case, this conjecture is still open. 
In this paper, we focus ourselves on polynomial maps $F:\C^2\rightarrow \C^2$, although one should be able to extend the ideas in this article to higher dimensional cases.

Now given a Jacobian map $F$, it is well-known that there are several conditions that are equivalent to $F$ being an automorphism. We will only the following two equivalent conditions \cite{bcw}:
\begin{itemize}
\item $F$ is proper.
\item $F$ is birational.
\end{itemize}

 In this article we try to understand $F$ by looking at the picture at infinity. To be more precise, we first extend $F$ to be a rational map $F:\CP^2\dashedrightarrow \CP^2$. Then we resolve the indeterminacy by blow-ups with blow-up centers supported away from $\C^2$. We denote the resulted morphism $\tilde{F}: \widetilde{\CP^2}\rightarrow \CP^2$. We denote by $p: \widetilde{\CP^2}\to \CP^2$ the projection, and by $H_{\infty}=\CP^2\backslash \C^2$ the line at infinity. If we can describe $\widetilde{\CP^2}\backslash \C^2=p^{-1}(H_{\infty})$ well enough, we should be able to prove(or disprove) the properness of $F$. 

It is an easy conclusion that the Jacobian of $F$ is constant from the assumption of the Jacobian conjecture. The main observation of this paper is that this conclusion means that $F$ induces an isomorphism between the spaces of integrable functions on the domain and target space of $F$ respectively, with the Lebesgue measure on both sides. And this isomorphism should give us more information about the behavior at infinity.

Let $L$ denote the proper transform of $H_{\infty}$ under the projection $p$, and by $E_i$ the exceptional curves of $p$. Then we can draw a graph, with $L$ and the $E_i$'s as nodes, and an edge for each intersection of the curves corresponding to the nodes.  We will call such a graph as "the graph of $F$".

Associated to each $E_i$, we have an integer, called the pole order of $E_i$, which will be introduced in the next section. We then have the following theorem, describing the graph of $F$ when $F$ is an automorphism.

\begin{theo}\label{auto}
When $F$ is an automorphism, then the graph can be described as following: The blow-ups are strictly in order, namely we can strictly order the indices of the $E_i$'s according to the order of blow-ups, with $E_1$ being the strict transform of the exceptional divisor of the first blow-up and the blow-up which produces $E_{i+1}$ is centered in $E_i$. Moreover, let $E_m$ be the last exceptional curve, then $E_m$ has pole order $6$.
\end{theo}

The converse is also true:

\begin{theo}\label{converse}
Let $F:\C^2\to\C^2$  be a polynomial map with non-vanishing Jacobian. If its graph satisfies the description of that for an automorphism in theorem \ref{auto}, then $F$ is an automorphism. In fact, we can weaken the condition by only requiring that $E_m$ has pole order  $\geq 0$, and the conclusion still holds.

\end{theo}

This theorem leads us to one direction to prove the Jacobian Conjecture: if we can show that the graph of any polynomial map of $\C^2$ to itself with non-vanishing Jacobian satisfy the condition described above, then we have affirmative answer to the Jacobian Conjecture for $\C^2$.

In section \ref{poleorders}, we will prove two technical lemmas: theorem \ref{order1} and theorem \ref{order2}, which we hope  will be helpful if one attempts to understand the graph better.

Along the same line of ideas, we are able to prove results in the following direction:

$\tilde{F}$  is a projective morphism, so by the Stein factorization theorem \cite{Hartshorne}, there exists a projective surface $Z$ such that we have the following commutative diagram

\begin{diagram}
\widetilde{\CP^2}& \rTo^{C} & Z\\
& \rdTo_{\tilde{F}}  & \dTo_{\pi}  \\
&      & \CP^2
\end{diagram}

\begin{Nota}
 Since the domain and target space of $F$ are both $\C^2$, in order to distinguish them, we use indices $X$ and $Y$ for the domain and target space respectively. So by $\C^2_X$ ($\C^2_Y$) we mean the domain (target space) of $F$. The same rule apply to the extensions of $F$.

\end{Nota}

We denote by $B=\pi^{-1}(\C^2_Y)$. So we always have $\C(\C^2_X)\subset B$. Since $C|_{\C^2_X}$ is injective, by abuse of notation, we write $\C^2_X\subset B$. Now, $F$ being proper is equivalent to $B=\C^2_X$. So it should be very useful if we can understand the open set $B$.  In this direction, we prove theorem \ref{singularity}. To state the theorem, first we need a definition.

\begin{defin}
A point on a normal surface is said to have $\tilde{A}_n$ singularity if the Dynkin diagram of the exceptional divisor of a minimal resolution is same as that of $A_n$ singularity, except that the self-intersection numbers are $\leq -2$ instead of being $-2$.
\end{defin}

\begin{theo}\label{singularity}
When $n=2$, we consider $\C^2$ as an open subset of $B$. Then $B\backslash \C^2$, if not empty, is a disjoint union of closed curves, whose normalizations are isomorphic to $\C$, and each of which has at most one singular point,  which is a cusp. Moreover $B$ has at most $\tilde{A}_n$ singularities, which locate at the singular points of the curve $B\backslash \C^2$.

\end{theo}

It is clear that $B\backslash \C^2$ is the union of the images of those $E_i$'s that are not mapped to $H_{\infty}\subset \CP^2_Y$. When $B$ is smooth along the image of such an $E_i$ under $C$, we can say more about $\tilde{F}$.
\begin{theo}\label{piecewise}
Let $E_i$ be an exceptional curve of $p$. If $C(E_i)\cap B$ is a curve, along which $B$ is smooth. Then the restriction map $\tilde{F}|_{E_i}$ is birational. And the only piecewise-singularity of $\tilde{F}(E_i)$ has local defining function $x^p=y^{p+1}$. And when the pole order of $E_i$ is $0$, there is no piecewise-singularity.

\end{theo} 

\begin{rem}
By piecewise-singularity of a curve we mean the following:

At a singular point of a curve, locally, we can decompose the curve into irreducible pieces analytically (the number of pieces equals the number of preimages of this point under the normalization map). Then we call the singularity on each component a piecewise-singularity of that curve. One sees that the local models of piecewise-singularities are given by $x^p=y^q$, where $x$ and $y$ are local holomorphic cooridnates and $p, q$ are coprime integers. 

\end{rem}

This article is structured as follows. In section \ref{poleorders}, we will develop our ideas by proving some technical lemmas. Then in section \ref{proofs}, we will first prove theorem \ref{singularity}. Then using the ideas in the proof of theorem \ref{singularity}, we prove theorem \ref{auto} and its converse theorem \ref{converse}. In the end, we will  prove theorem \ref{piecewise} and give some interesting examples. It is interesting but not surprising that in the proof of theorem \ref{piecewise}, we use the idea of singularity exponent from \cite{DeKo}, since we are using integrability of functions in this article.

\textbf{Acknowledgement}: The author would like to thank Professor Shiffman for his uninterrupted unconditional supports.

\section{Pole Orders}\label{poleorders}

For now, let us go back to $\C^n$. The Euclidean measure on $\C^n$ can be seen as a measure on $\CP^n$ or $\widetilde{\CP^n}$ by just setting the complement to carry $0$ measure. We denote the measure by $d\mu$, and by $d\mu_X$ and $d\mu_Y$ the corresponding measures on the domain and target space of $\tilde{F}$ respectively. 
 A basic observation is the following: 

\begin{lem}
Let $f$ be a function on $\CP^n$ with measure $d\mu_Y$, $\tilde{f}=f\circ \tilde{F}$ the pull back of $f$ to $\widetilde{\CP^n}$ with measure  $d\mu_X$. Then $f$ is $L^1$ if and only if $\tilde{f}$ is $L^1$. In fact, we always have $\int |\tilde{f}|d\mu_X=j\int |f|d\mu_Y$, where $j$ is the degree of $\tilde{F}_*: H_{2n}(\widetilde{\CP^n}, \Z)\to H_{2n}(\CP^n,\Z)$.

\end{lem}
\begin{proof}
$\tilde{F}$ is a branched covering. So by deleting the branched sets from the domain and image, which are of measure 0, we get a covering map, which is of degree $j$. As noticed in the introduction, the volume form is preserved by $\tilde{F}$. So the lemma follows. 
\end{proof}

The local version is the following:
\begin{lem}  \label{main} 
Let $\tilde{F}(x)=y$, $f$ a function in a small neighborhood of $y$. 
\begin{itemize}
\item[a] If $\tilde{F}$ is finite in a neighborhood $x$, then $f$ is locally $L^1$ at $y$ if and only if $f\circ \tilde{F}$ is locally $L^1$ at $x$.
\item[b] Without the finite assumption, if  $f$ is locally $L^1$ at $y$ then $f\circ \tilde{F}$ is locally $L^1$ at $x$.
\end{itemize}

\end{lem}

\begin{proof}
Part (b) follows directly from lemma \ref{main}. 

With the assumption of part (a),
$\tilde{F}$ is finite at a generic point $p$ of $E$, so by the open mapping theorem \cite{open}, the image a small neighborhood $U$ of $p$ contains a neighborhood $V$ of $q=\tilde{F}(p)$. Then together with lemma \ref{main}, we get part (a).

\end{proof}

The measure on $\CP^n$ is not hard to understand. In the $\C^n$ part, it is just the Euclidean measure given by $\lambda=\frac{1}{n!}\omega^n$, where $\omega=\frac{\sqrt{-1}}{2}\sum_{i=1}^ndz_i\wedge d\bar{z}_i$. Around $H_{\infty}=\CP^n\backslash \C^n $, we can calculate directly to see that $d\mu$ is of the form $\frac{1}{|y|^{2n+2}}\frac{1}{n!}\omega^n$, where $y$ is a local defining function of $H_{\infty}$. Coming to $\widetilde{\CP^n}$, we need a sequence of blow-ups. Then the exact formulas of $d\mu$ is complicated to calculate. But when $n=2$, the surface case, we have  an easier description. 

In the following, we will assume $n=2$.

Now $P:\widetilde{\CP^2}\to \CP^2$ is just a composition of a sequence of monoidal transformation. In order to do the calculation, we consider $P:\widetilde{\CP^2}\to \CP^2$ as a process of blow-ups, with all the blow-up centers lying on the exceptional curves or the strict transforms of $L=\CP^2\backslash \C^2$. For simplicity, we will call the strict transforms of $L$ together with the exceptional curves the t-curves, all of which are copies of $\CP^1$.
During the process, there are two types of blow-ups. One is called a single blow-up, meaning the blow-up center does not lie on the intersections of the t-curves. The other type is called a double blow-up, meaning that the blow-up center lie on  the intersections of the t-curves, with the understanding that no three t-curves intersect at a single point.

 We start the calculation with a single blow-up at $p\in L$. Let $(x,y)$ be local coordinate at $p$, with $L=(y=0)$, meaning that $y$ is the defining function of $L$. Then  $d\mu=\frac{1}{|y|^6} \lambda$. We write the blow-up locally as $Bl_p:\tilde{U}\to U$, $E=Bl_p^{-1}(p)$ the exceptional curve, and $\tilde{L}$ the strict transform of $L$. Let $q=E\cap \tilde{L}$. Then around any point $\tilde{L}\backslash q$, the form of $d\mu$ remain the same. 
 
  Let $y=xz$, then around $q$, $(x,z)$ are local coordinates, with $\tilde{L}=(z=0)$, and $E=(x=0)$. So $Bl_p^*(d\mu)=\frac{\lambda}{|x|^4|z|^6}$. At a different point $o\in E$, let $x=yz$, then $(y,z)$ are local coordinates, with $E=(y=0)$. And  $Bl_p^*(d\mu)=\frac{\lambda}{|y|^4}$. One should be convinced immediately by the following conclusion:

  \begin{lem}\label{single}
  Let $E_0$ be a t-curve. Let $Bl_p$ be a single blow-up with $p\in E_0$. And suppose around $p$, $d\mu=\frac{\lambda}{|y|^m}$ with $E_0=(y=0)$. Let $E_1=Bl_p^{-1}(p)$, $\tilde{E_0}$ the strict transform of $E_0$, and $q=\tilde{E_0}\cap E_1$. Then 
  \begin{itemize}
  \item[a]Around $q$,  $Bl_p^*(d\mu)=\frac{\lambda}{|x|^{m-2}|z|^m}$, where $(x,z)$ are local coordinates with $\tilde{E_0}=(z=0)$ and $E_1=(x=0)$.
  
  \item[b] Around a point $o\in E_1\backslash \{q\}$, $Bl_p^*(d\mu)=\frac{\lambda}{|y|^{m-2}}$, where $(y,z)$ are local coordinates with $E_1=(y=0)$.
  \end{itemize}
  
  \end{lem}

Next, we consider a double blow-up. Let $p=E_0\cap E_1$, with local coordinates $(x,y)$, and $E_0=(x=0)$, $E_1=(y=0)$. We assume that $d\mu=\frac{\lambda}{|x|^m|y|^n}$. Then after blowing-up, we get a new t-curve $E_2$, which intersects $\tilde{E_0}$ at $q$, and  intersects $\tilde{E_1}$ at $o$. Now around $q$, we have local coordinates $(y,z)$, with $x=yz$, and $\tilde{E_0}=(z=0)$, $E_2=(y=0)$. So around $q$, $Bl_p^*(d\mu)=\frac{\lambda}{|y|^{m+n-2}|z|^m}$. Around $o$, we have local coordinates $(x,w)$, with $y=xw$, and $\tilde{E_1}=(w=0)$, $E_2=(x=0)$. So around $q$, $Bl_p^*(d\mu)=\frac{\lambda}{|x|^{m+n-2}|w|^n}$. And for any other point in $E_2\backslash \{q,o\}$, if $E_2=(t=0)$, then around $q$, $Bl_p^*(d\mu)=\frac{\lambda}{|t|^{m+n-2}}$.

With the understanding of the calculations above, we can define

\begin{defin}
We say a t-curve $E$ has pole order $m$, if around a generic point (different from those intersections with other t-curves), $d\mu=\frac{\lambda}{|y|^m}$, where $y$ is the local defining function of $E$, written as $E=(y=0)$.

\end{defin}

One sees that in this way we have defined a pole order for each t-curve. And since the process of blowing-up start from $L$ which has pole order 6, the pole orders of all t-curves are even. We will call the part $\frac{1}{|y|^m}$ in $d\mu$ the principle of $d\mu$. The pole order also generalize to the points in $\C^2$, which of course are 0. At the intersections of t-curves, the principle of the volume form simply is the product of the principles of the intersecting t-curves. We can now summarize as follows

\begin{theo}\label{pole}With the notations as above,

\begin{itemize}
\item[a] If we have a single blow-up at a point $p$ in a t-curve $E$ with pole order $m$, then the new t-curve $Bl_p^{-1}(p)$ has pole order $m-2$.

\item[b] If we have a double blow-up at a point $p=E_0\cap E_1$ with pole order $m$ and $n$ respectively, then the new t-curve $Bl_p^{-1}(p)$ has pole order $m+n-2$.
\end{itemize}

\end{theo}

Again, we consider the morphism $\tilde{F}:\widetilde{\CP^2}\to \CP^2$. Let $T$ denote the pull back $\tilde{F}^*(H_{\infty})$ of the divisor $H_{\infty}\subset \CP^2$. Then $T=\sum a_iE_i$, with $a_i$ positive integers, and $E_i$ t-curves. The support of $T$ may not contain all t-curves. For now, each t-curve $E_i$ in the support of $T$ carries a data with 3 numbers (p,m,s), where $p$ is the pole order,  $m=a_i$ the multiplicity, and $s=E_i^2$ the self intersection number. Now we can apply lemma \ref{main} and get the following theorem.

\begin{theo}\label{order1}
Let $E$ be a t-curve in the support of $T$, with pole order $p$ and multiplicity $m$

\begin{itemize}
\item[a] If $\tilde{F}(E)=H_{\infty}$, then  $m=\frac{p-2}{4}$.

\item[b] If $\tilde{F}(E)$ is a point, then $m\geq \frac{p-2}{4}$
\end{itemize}

\end{theo}

\begin{proof} 
At any point in $H_{\infty}$, the pole order of the measure is 6. Let $y$ be the local defining function of $H_{\infty}$, then the function $f=|y|^t$ is locally $L^1$ if and only if $t\geq 4$. 

Under the assumption of part (a),   the pull back  of $f$ is $\tilde{f}\approx |x|^{mt}$, where $x$ is the local defining function of $E$. So $\tilde{f}$ is locally $L^1$ if and only if $mt>p-2$. So by lemma \ref{main}, we see that $m=\frac{p-2}{4}$.

Under the assumption of part (b), let $q=\tilde{F}(E)$. When $t> 4$, the function $f=|y|^t$ is locally $L^1$ around $q$, so by lemma \ref{main} again, we get that $\tilde{f}=f\circ \tilde{F}\approx |x|^{mt}$ is locally $L^1$ along $E$. So when $t>4$, we have $mt>p-2$. We conclude that $m\geq \frac{p-2}{4}$.

\end{proof}

Consider the t-curves that are not contained in the support of $T$, namely those t-curves whose images under $\tilde{F}$ intersect $\C^2$. By lemma \ref{main}, the pole order of these t-curves must be $\leq 0$.
For these t-curves, we can also define the multiplicity as following:
\begin{itemize}
\item When $E_i$ is mapped onto a curve $C$, we can pull back a defining function of $C$ at a smooth point of $C$. Then we define the multiplicity of $E_i$ as the vanishing order of that function along $E_i$. It is easily seen that this definition does not depend on the choice of smooth point, nor on the choice of local defining function.
 \item  When $E_i$ is contracted to a point $p$, we can define the multiplicity of $E_i$ as the maximum of the vanishing order of the pull back of all functions that have multiplicity $1$ at $p$.
\end{itemize}

We can apply argument similar to the proof of theorem \ref{order1} to these curves, and get the following theorem.

\begin{theo}\label{order2}
Let $E$ be a t-curve that is not contained in the support of $T$, with pole order $p$ and multiplicity $m$.
\begin{itemize}
\item[a] If $\tilde{F}(E)$ is a curve, then  $m=\frac{2-p}{2}$.

\item[b] If $\tilde{F}(E)$ is a point, then $m\leq \frac{2-p}{2}$
\end{itemize}
\end{theo}
\begin{proof}
Basically same as that of theorem \ref{order1}. Let $f$ be any function that has multiplicity at $p$. Then $\frac{1}{|f|^c}$ is locally $L^1$ around $p$ if and only if $c<2$. Let $m_f$ be the vanishing order of the pull back of $f$ along $E$. Then as in the proof of theorem \ref{order1}, we get $m_f\leq \frac{2-p}{2}$. Hence taking the maximum, we get the theorem.

\end{proof}

 We associate to  each t-curve a vertex, to each intersection an edge, then we get a connected graph. We consider the strict transform $\tilde{L}$ of $L$ the root of this graph, and assign to each edge the "away from root" direction, then this graph become a rooted tree, which we denote by $G$. We adapt the terminologies such as child, parent from the graph theory. 

Given a vertex $E$ in $G$, by children of $E$, we mean the vertices whose paths to the root have to pass $E$, and by direct children  the children which are adjacent to $E$. Beside the root $\tilde{L}$, a vertex $E$ has two types of parents: one is the parent along the tree, which we call its t-parent. The other type of parents, which we call the b-parents, of $E$ are those t-curves that $E$ were blown-up during the process of blowing-ups. Here we are abusing notation by identifying curves with their strict transforms.

In order to understand the structure, we need to analyze the pole orders along the tree. 
\begin{defin}
We call a pair of integers $(a,b)$ a pole pair if there exists a pair of adjacent vertices $(E_1, E_2)$ on $G$, with $E_1$ the t-parent of $E_2$, such that $a$ is the pole order of $E_1$ , $b$ the pole order of $E_2$.

\end{defin}\label{pair}

\begin{lem}Let  $(a,b)$ be a pair of integers.
\begin{itemize}
\item[a] If $a\leq 2$ and $b\geq 2$, then it is not a pole pair.

\item[b] If $a\geq 2$ and $b\leq -2$, then it is not a pole pair.
\item[b'] If $a\geq 4$ and $b\leq 0$, then it is not a pole pair.
\item[c]  If $a=b\leq -2$, then it is not a pole pair.
\end{itemize}

\end{lem}
\begin{proof}
We need to look at the blow-up process to find blow-ups that possibly produce such pairs.

Under the assumption of each situation, if such a pair appear in $G$, it must come from a double blow-up. So in one intermediate step, there is a pair of intersecting t-curves $(E_1,E_2)$ with pole pair $(c,d)$. Let $p=E_0\cap E_1$, then we assume we get the pole pair $(a,b)$ by blowing up at $p$. By theorem \ref{pole}, we see that $(a,b)=(c, c+d-2)$ or $(a,b)=(c+d-2,d)$.

Now for part (a), if $(a,b)=(c, c+d-2)$, then $c\leq 2$ and $c+d-2\geq 2$, we get $d\geq 2$. If $(a,b)=(c+d-2,d)$, then $c+d-2\leq 2$ and $d\geq 2$ and we get $c\leq 2$. In both cases,  $(c,d )$ is already  a pole pair that satisfy the assumption of (a). But the start-up pair is $(6,4)$, so pairs satisfying the assumption can never be produced. And part (a) is proved. 

For part (b),  if $(a,b)=(c, c+d-2)$, then $c\geq 2$ and $c+d-2\leq -2$, we get $d\leq -2$. If $(a,b)=(c+d-2,d)$, then $c+d-2\geq 2$ and $d\leq -2$ and we get $c\geq 6$.  In both cases,  $(c,d )$ is already  a pole pair that satisfy the assumption of (b). But the start-up pair is $(6,4)$, so pairs satisfying the assumption can never be produced. And part (b) is proved. Part (b') can be proved in the same way.

For part (c),
 if $(a,a)=(c, c+d-2)$, then $c\leq -2$ and $d=2$, contradicting part (a). If $(a,a)=(c+d-2,d)$, then $d\leq -2$ and $c=2$, contradicting part (b). So part (c) is proved. 
\end{proof}

\section{proofs of the theorems}\label{proofs}

We will first prove theorem \ref{singularity}. Theorem \ref{auto} and its converse theorem \ref{converse} will then follow the same ideas.
\begin{proof}(theorem \ref{singularity})

Let $E_0$ be a t-curve whose image is a curve intersecting $\C^2$. By lemma \ref{main}, the pole order of $E_0$ must be $\leq 0$. Then by lemma \ref{pair}, all the children of $E_0$ have pole orders $\leq 0$.
We claim that the children of $E_0$ must contract to points in $\C^2$.

 To see this, let $E_1$ be a child, not necessarily direct, of $E_0$. Assume $C_1=\tilde{F}(E_1)$ is not a point, then $C_1$ intersects $H_{\infty}$ at, say, $A$. Let $p\in E_1$ satisfying $\tilde{F}(p)=A$, $q=C(p)\in Z$, where $C:\widetilde{\CP^2}\to Z$ is the birational morphism in the Stein Factorization. Consider the connected set $C^{-1}(q)$. If $C^{-1}(q)=p$, then the image of any small neighborhood of $p$ would contain an open neighborhood of $q$, hence its image under $\tilde{F}$ would contain a neighborhood of $A$. But a nonzero constant function is not integrable around $A$, but the pull back of this function will be integrable around $p$, contradicting lemma \ref{main}. So $C^{-1}(q)$ is a connected curve passing $p$, all of whose components are children of $E_0$. So the image of any small neighborhood of $C^{-1}(q)$ would contain an open neighborhood of $q$, hence its image under $\tilde{F}$ would contain an open neighborhood of $A$. Then again, we use nonzero constant function to get a contradiction. So the claim has been proved.
 
 We then consider the branches of $E_0$, forming subtrees, each of which has as its root one of the direct children of $E_0$. Notice that among the children of $E_0$, there may be a b-parent of $E_0$, in which case there maybe b-grandparent(great grandparent, etc.). But all these ancestors can only be contained in one subtree. So all other subtrees consist of b-children of $E_0$. Consider one of such subtrees, it must contract to one single point in $\C^2$ under $\tilde{F}$, since it is connected. Then we see that these subtrees are redundant. So if we assume that $\tilde{F}:\widetilde{\CP^n}\to \CP^n$ is a minimal resolution, namely, no $-1$ curve, except the strict transform of $L$, contracts to a point under $\tilde{F}$, then $E_0$ can only have one subtree, which contain a b-parent of $E_0$. Similar argument as before can be applied to get that this subtree contains no b-children of $E_0$. Then we argue again on the braches of the b-parent of $E_0$ which is a t-child of $E_0$ to get that it has at most one direct child. We can keep arguing like this to conclude that  $E_0$ together with its t-children form a line-like tree, namely each vertex has at most one direct child. Moreover, each vertex in this tree has one of its b-parent as its only direct t-child. Therefore the t-children contract to a point in $B$ with at most $\tilde{A}_n$ singularity. 
 
 Let $E_{-1}$ be the t-parent of $E_0$. Then by lemma \ref{pair}, the pole order of $E_{-1}\leq 2$. So $E_{-1}$ contracts to a point in $H_{\infty}$. So $E_0\backslash E_{-1}\cong \C$ maps injectively to $B$. And the image has at most one singularity, which must be a cusp.
 So we have proved theorem \ref{singularity}
\end{proof}
We now prove theorem \ref{auto}.

\begin{proof}(theorem \ref{auto})
By assumption, $F$ is an automorphism. So among all the t-curves, there is only one, called $E_f$, that is mapped to $\CP^2_Y\backslash \C^2_Y$ finitely by $\tilde{F}$. Moreover the multiplicity of $\tilde{F}$ along $E_f$ must be 1. Therefore the pole order of $E_f$ must be $6$. 

So  all other $E_i$ are contracted by $\tilde{F}$, hence having self-intersection number $\leq -2$. Therefore there must be exactly one stream of blow-ups, namely, each blow-up is centered in the exceptional curve of the immediate previous blow-up, otherwise we should have more than one curve with self-intersection number $-1$. So we can strictly order the indices of the $E_i$'s as asserted in the theorem. And then $E_f$ must have self-intersection number $-1$, hence being the exceptional curve produced in the last blow-up. So we have proved theorem \ref{auto}.

\end{proof}
\begin{rem}
It was proved in \cite{abhyankar}, the group of automorphisms of $\C^2$ is generated by the linear automorphisms and elementary non-linear automorphisms $F_n=(x+y^n,y)$ . So it is possible to prove theorem \ref{auto} by using this fact.

\end{rem}

\begin{proof}(theorem \ref{converse})
By the assumption of the theorem, the blow-ups which resolve $F$ are strictly ordered, and $E_m$ has pole order $p\geq 0$. If $p\geq 2$, according to lemma \ref{pair}, we see that every $E_i$ has pole order $\geq 2$, since once an $E_i$ with pole order $0$ is produced, we would not be able to get $E_j$ of pole order $\geq 2$ by further blow-ups. So by theorem \ref{order2}, every $E_i$ is contained in the support of $T$, implying that $F$ is a proper map, hence $F$ is an automorphism.

 Now we assume that $p=0$. Then by lemma \ref{pair} again, every $E_i$ has pole order $\geq 0$.
 
 Assume $F$ is not proper, then some $E_i$'s are mapped onto some curves other than $H_{\infty}\subset \CP^2_Y$. Then by theorem \ref{order2}, these $E_i$ are of pole order $\leq 0$. Therefore they are all of pole order  $0$. Then by theorem \ref{order1}, they are all of multiplicity $1$. This implies that $\tilde{F}|_B$ is not branched  over a generic point of any of these $E_i$'s. So  $\tilde{F}|_B: B\to \C^2_Y$ is branched over finite points $\{p_i\}_{1\leq i\leq n}$. By deleting these points from $C^2_Y$ and their pre-image from $B$, we get a covering map. But since $\C^2_Y\backslash\{p_i\}_{1\leq i\leq n}$ is simply connected, we see that $\tilde{F}$ is a birational map. But this also imply that $F$ is an automorphism \cite{bcw}. We have proved the theorem.
\end{proof}

\begin{proof}
(theorem \ref{piecewise})

We will consider the part of $E_i$ that is mapped to $B$ by $C$. But by theorem \ref{singularity}, that part is the whole $E_i$ except $1$ point. By abuse of notation, we will call this part also $E_i$. 

Now let $o$ be a point in $E_i$, at which we have local coordinates $(x,y)$ and $E_i$ is locally given by $y=0$. By the assumption of the theorem, we see that $\tilde{F}$ is finite at $o$, following theorem \ref{singularity}.

First, assume that the image of a local piece $E_{i, o}$ of $E_i$ around $o$ under $\tilde{F}$ is smooth in $\C^2_Y$. So we can choose local coordnates $(z,w)$ around $\tilde{F}(o)$ such that $\tilde{F}(E_{i, o})$ is defined by $z=0$. Let $m$ and $p$ be the multiplicity and pole order of $E_i$. Then $\tilde{F}$ is defined by $\tilde{F}(x,y)=(y^mg(x,y), f(x,y))$, where $y\nmid g$ and $f(0,0)=0$. 

\begin{claim}
$g(0,0)\neq 0$.
\end{claim}

To see this, consider the integral $\int \frac{d\mu}{|zw|^{2c}}$. This converges if and only if $c<1$. So the integral $I=\int \frac{|y|^{-p}d\mu}{|y|^{2mc}|gf|^{2c}}$ converges if and only if $c<1$. By theorem \ref{order2}, $-p=2m-2$. So the integral 
$I=\int \frac{d\mu}{|y|^{2mc-2m+2}|gf|^2c}$. Assume $g(0,0)=0$, then the singularity exponent (we mention singularity exponent here without the definition just to interest the reader, the meaning will be explained right away) of $ygf$ must be less than $1$, namely there exists $c_0<1$ such that whenever $c>c_0$, we have $\int \frac{d\mu}{|ygf|^2c}$ divergent. So if we pick $0<\epsilon$ small enough, and let $c=1-\epsilon$ to make $2mc-2m+2>2c_0$ and $c>c_0$, we get a contradiction. And the claim is proved.

\smallskip

Now that $g(0,0)\neq 0$, we can take the $m-$th root. Let $h(x,y)$ be one branch of $g^{1/m}$. Then locally $\tilde{F}$ factor as $\tilde{F}=S_m\circ H$, where $H(x,y)=(yh,f)$, and $S_m(x,y)=(x^m,y)$. We claim that $H$ is  biholomorphic, namely that the Jacobian of $H$ at $o$ is non-zero. We can calculate the Jacobian directly:
$$J_H=\left[\begin{array}{rr}yh_x & g+yh_y\\f_x & f_y  \end{array}\right]=yf_yh_x-hf_x-yf_xh_y$$
Since the Jacobian of $\tilde{F}$ can not vanish along $E_i$, $J_H$ can vanish only along $(y=0)$. Assume $y|J_H$, then $y|hf_x$. But since $h(0,0)\neq 0$, we must have $y|f_x$. Since $f(0,0)=0$, we must then have $y|f$. So $E_i=(y=0)$ is contracted by $\tilde{F}$ to a point, a contradiction. So $H$ is biholomorphic. So after a change of coordinates $\tilde{F}$ is modeled as the standard $S_m(x,y)=(x^m,y)$. In particular, $\tilde{F}|_{E_i}$ has non-vanishing tangent map. We can also conclude that $J_{\tilde{F}}\sim y^{m-1}$.

Next assume $S=\tilde{F}(E_{i, o})$ is singular at $\tilde{o}$. So there are local coordinates such that $S$ is given by $z^a=w^b$ in a neighborhood $V$, with $a$ and $b$ being coprime and $a<b$. The normalization of $S$ is given by $\pi_0: U\to V$ with $\pi_0(t)=(t^b, t^a)$, where $U$ is a neighborhood of $0$ in $\C$.
The map $\tilde{F}|_{E_i}$ lifts to the normalization of $\tilde{F}(E_i)$, which in local coordinates can be written as $x\mapsto x^t$, where $t$ is a positive integer. So now $\tilde{F}|_{E_i}$ can be written as $x\mapsto (x^{tb}, x^{ta})$. Therefore $\tilde{F}(x,y)=(x^{tb}+f_1, x^{ta}+f_2)$, with $y|f_i$ for $i=1, 2$.
So we can calculate the Jacobian
$$J=J_{\tilde{F}}=\left[\begin{array}{rr}tbx^{tb-1}+f_{1,x} & f_{1,y}\\tax^{ta-1}+f_{2,x} & f_{2,y}  \end{array}\right]=tbx^{tb-1}f_{2,y}+f_{1,x}f_{2,y}-tax^{ta-1}f_{1,y}-f_{2,x}f_{1,y}$$
From the previous argument, we know that $J=y^{m-1}g(x,y)$ for some holomorphic $g$ satisfying $g(0,0)\neq 0$.
 
When $m=1$, we have $y|f_{i, x}$.  Now in order for the Taylor expansion of $J$ to have a constant term  $e$ with $e\neq 0$. We must have $ta-1=0$. The only possibility is  $t=1$ and $a=1$. But then $S=\tilde{F}(E_{i, o})$ is smooth. 

Now, we assume $m>1$. 
Assume that $y^{m-1}|f_{i,y}$ for $i=1, 2$. Then since $y|f_i$, we see that $y^m|f_i$ for $i=1, 2$. But then $y^m|f_{i,x}$, then we can not find a term of the form $ey^{m-1}$, with $e\neq 0$ constant, in the Taylor expansion of $J$, a contradiction. Then since we must have $y^{m-1}|(tbx^{tb-1}f_{2,y}-tax^{ta-1}f_{1,y})$, we have $\lambda=\deg_y f_1=\deg_y f_2<m-1$. Now we expand $f_{i,y}$ as Taylor series of $y$, and denote the coefficients by $c_{i,\lambda}(x)$, then we have $ac_{1,\lambda}=x^{t(b-a)}bc_{2,\lambda}$ for $\lambda<m-1$. Now in order for the Taylor expansion of $J$ to have a term of the form $ey^{m-1}$, with $e\neq 0$ constant, we must have $t(b-a)=1$. The only possibility is  $t=1$ and $b=a+1$. We have proved the part about the piecewise-singularity asserted in the theorem. 

Now globally, $\tilde{F}|_{E_i}$ lifts to the normalization of $\tilde{F}(E_i)$. We denote the lifted map by $G$. Then by what we have shown above, $G$ can not have critical point on $E_i\cap C^{-1}(B)$. This means that $G$ can have at most 1 critical point. But both $E_i$ and the image are copies of $\CP^1$. So $G$ must be a linear map $\CP^1\to \CP^1$. Therefore $\tilde{F}|_{E_i}$ is birational. We have proved the theorem. 

\end{proof}
\begin{rem}
The $c_0$ in the proof in fact \cite{DeKo} satisfies $c_0\leq \frac{2}{3}$ .

\end{rem}
\begin{example}
If we only consider the local case, the piecewise-singularity $x^p=y^{p+1}$ is actually possible. Consider the following map
$G(x,y)=(2x^3+xy, 3x^2+y)$. The Jacobian is $J=y$, corresponding to $m=2$. And the image of the curve $y=0$ is defined by $27x^2=4y^3$.

\end{example}

\begin{example} For elementary non-linear automorphsim of $\C^2$ of the form $F_n=(x+y^n,y)$, we can draw out the graph  with explicit ordering of the nodes as follows:

\begin{tikzpicture}
\matrix [column sep=7mm, row sep=5mm] {
  \node (l) [draw, shape=circle] {L}; &
  \node (e2) [draw, shape=circle] {$E_2$}; &
    \node (e3) [draw, shape=circle] {$E_3$}; & 
     \node (d) [draw, shape=circle] {$\cdots$}; &
    \node (en) [draw, shape=circle] {$E_n$}; &
     \node (e1) [draw, shape=circle] {$E_1$}; &
  \\
 & & & & 
  \node (en1) [draw, shape=circle] {$E_{n+1}$}; & \\
  & & & &   \node (dd) [draw, shape=circle] {\vdots}; &
    \\
   & & & &   \node (e2n) [draw, shape=circle] {$E_{2n-1}$}; 
   \\
};
\draw[thick] (l) -- (e2);
\draw[thick] (e2) -- (e3);
\draw[thick] (e3) -- (d);
\draw[thick] (d) -- (en);
\draw[thick] (en) -- (e1);
\draw[thick] (en) -- (en1);
\draw[thick] (en1) -- (dd);
\draw[ thick] (dd) -- (e2n);
 \node[below right] at (current bounding box.south west){\textbf{Graph 2:}$ F=(x+y^n,y)$};
\end{tikzpicture}

\end{example}

\bibliography{jacobian} 
\bibliographystyle{plain}

\end{document}